\documentclass[envcountsect,referee]{svjour3}

\smartqed  
\usepackage{graphicx}
\usepackage{latexsym,bm,amsmath,amssymb,mathrsfs}
\usepackage{rotating}
\usepackage{multirow,booktabs,cases}
\usepackage[misc]{ifsym}
\usepackage[style=1]{mdframed}
\usepackage{algorithm,algorithmic}
\usepackage[colorlinks=true]{hyperref}
\hypersetup{urlcolor=blue,citecolor=blue,linkcolor=blue}
\usepackage{epstopdf}

\def\[{\begin{equation}}
\def\]{\end{equation}}
\def\lb{\left(}
\def\rb{\right)}

\def\t{\top}
\def\A{{\mathcal A}}
\def\B{{\mathcal B}}

\numberwithin{equation}{section}
\allowdisplaybreaks

\begin{document}
\graphicspath{{./FIG/}}

\title{On the cone eigenvalue complementarity problem for higher-order tensors}


\author{Chen Ling \and Hongjin He \and Liqun Qi}


\institute{C. Ling\and H. He \at
Department of Mathematics, School of Science, Hangzhou Dianzi University, Hangzhou, 310018, China.\\
\email{cling\_zufe@sina.com}
\and H. He (\Letter) \at
\email{hehjmath@hdu.edu.cn}
\and L. Qi \at
Department of Applied Mathematics, The Hong Kong Polytechnic University, Hung Hom, Kowloon, Hong Kong. \\
\email{maqilq@polyu.edu.hk}
 }

\date{Received: date / Accepted: date}

\maketitle

\begin{abstract}
In this paper, we consider the {\it tensor generalized eigenvalue complementarity problem} (TGEiCP), which is an interesting generalization of matrix {\it eigenvalue complementarity problem} (EiCP). First, we given an affirmative result showing that TGEiCP is solvable and has at least one solution under some reasonable assumptions. Then, we introduce two optimization reformulations of TGEiCP, thereby beneficially establishing an upper bound of cone eigenvalues of tensors. Moreover, some new results concerning the bounds of number of eigenvalues of TGEiCP further enrich the theory of TGEiCP. Last but not least, an implementable projection algorithm for solving TGEiCP is also developed for the problem under consideration. As an illustration of our theoretical results, preliminary computational results are reported.



\end{abstract}

\section{Introduction}\label{Introd}
The complementarity problem has become one of the most well-established disciplines within mathematical programming \cite{FP03}, in the last three decades. It is not surprising that the complementarity problem has received much attention of researchers, due to its widespread applications in the fields of engineering, economics and sciences. In the literature, many theoretical results and efficient numerical methods were developed, we refer the reader to \cite{FP97} for an exhaustive survey on complementarity problems.

The {\it eigenvalue complementarity problem} (EiCP) not only is a special type of complementarity problems, but also extends the classical eigenvalue problem which can be traced back to more than 150 years (see \cite{GV00,VG97}). EiCP first appeared in the study of static equilibrium states of mechanical systems with unilateral
friction \cite{CFJM01}, and has been widely studied \cite{AR13,CS10,JRRS08,JSR07,JSRR09} in the last decade. Mathematically speaking, for two given square
matrices $A,B\in \mathbb{R}^{n\times n}$, EiCP refers to the task of finding a scalar $\lambda\in \mathbb{R}$ and a vector $x\in \mathbb{R}^n\backslash\{0\}$ such that
$$
0\leq x\perp w:=(\lambda B-A)x\geq 0.
$$
EiCPs are closely related to a class of differential inclusions with nonconvex processes defied by linear complementarity conditions, which serve as
models for many dynamical systems. Given a linear mapping $A:\mathbb{R}^n\rightarrow \mathbb{R}^n$, consider a dynamic system of the form:
\begin{equation}\label{MEDS}
\left\{
\begin{array}{l}
u(t)\geq 0,\\
\dot{u}(t)-Au(t) \geq 0,\\
\langle u(t),\dot{u}(t)-Au(t)\rangle=0.
\end{array}
\right.
\end{equation}
It is obvious that \eqref{MEDS} is equal to $\dot{u}(t)\in F(u(t))$, where the process $F : \mathbb{R}^n\rightarrow \mathbb{R}^n$ is given by
$${\rm Gr}( F):=\left\{(x,y)\in \mathbb{R}^n\times \mathbb{R}^n~|~ x \geq 0,  y-Ax \geq 0, \langle x, y-Ax\rangle= 0\right\}$$
and is nonconvex. As noticed
already by Rockafellar \cite{R79}, the change of variable
$u(t)= e^{\lambda t}v(t)$ leads to an equivalent system
$$
\lambda v(t)+\dot{v}(t)\in F(v(t)).
$$
This transformation efficiently utilizes the positive homogeneity of $F$. Therefore, if the pair $(\lambda, x)$ satisfies $\lambda x\in F(x)$, then the trajectory
$t\mapsto e^{\lambda t}x$ is a solution of dynamic system \eqref{MEDS}. Moreover, if such a trajectory is nonconstant, then $x$ must be a nonzero vector, which further implies that $(\lambda,x)$ is a solution of EiCP with $B:=I$ (i.e., $B$ is the identity matrix). The reader is referred to \cite{CFJM01,See99} for more details.

When $B$ is symmetric positive definite and $A$ is symmetric, EiCP is symmetric. In this case, it is well analyzed in \cite{QJH04} that EiCP is equivalent to finding a stationary point of a generalized Rayleigh quotient on a simplex. Generally speaking, the resulting equivalent optimization formulation is NP-complement \cite{C89,QJH04} and very difficult to be solved efficiently, and in particular when the dimension of the problem is large.

In the current numerical analysis literature, considerable interest has arisen in extending concepts
that are familiar from linear algebra to the setting of multilinear algebra. As a natural extension of the concept of matrices, a tensor, denoted by $\mathcal{A}$, is a multidimensional array, and its order is the number of dimensions. Let $m$ and $n$ be positive integers. We call
$\mathcal{A}= (a_{i_1\cdots i_m} )$, where $a_{i_1\cdots i_m} \in \mathbb{R}$ for $1\leq i_1, \ldots, i_m\leq n$, a real $m$-th order $n$-dimensional square tensor. The tensor $\mathcal{A}$ is further called symmetric if its entries are invariant under any permutation of their indices. The eigenvalues and eigenvectors of such square
tensor were introduced by Qi \cite{Q05}, and were introduced independently by Lim \cite{Lim05}.

For a vector $x = (x_1, \ldots,x_n)^\top\in \mathbb{C}^n$,  $\mathcal{A}x^{m-1}$ is an $n$-vector with its $i$-th component defined by
$$
(\mathcal{A}x^{m-1})_i=\sum_{i_2,\ldots,i_m=1}^na_{ii_2\cdots i_m}x_{i_2}\cdots x_{i_m},~~{\rm for~}i=1,2,\ldots,n,$$
and $\mathcal{A}x^m$ is a homogeneous polynomial defined by
$$\mathcal{A}x^{m}=\sum_{i_1,i_2,\ldots,i_m=1}^na_{i_1i_2\cdots i_m}x_{i_1}x_{i_2}\cdots x_{i_m}.$$
For given tenors $\mathcal{A}$ and $\mathcal{B}$ with same structure, we say that $(\mathcal{A},\mathcal{B})$ is an identical singular pair, if
$$
\left\{x\in \mathbb{C}^n\backslash\{0\}:\mathcal{A}x^{m-1}=0,\mathcal{B}x^{m-1}=0\right\}\neq\emptyset.
$$

\begin{definition} (\cite{CPZ09}) Let $\mathcal{A}$ and $\mathcal{B}$ be two $m$-th order $n$-dimensional tensors on $\mathbb{R}$. Assume that $(\mathcal{A},\mathcal{B})$ is not an identical singular pair.
We say $(\lambda, x)\in \mathbb{C}\times(\mathbb{C}^n\backslash\{0\})$ is an eigenvalue-eigenvector of $(\mathcal{A}, \mathcal{B})$, if the $n$-system of equations:
\begin{equation}\label{ABEigen}
(\mathcal{A}-\lambda \mathcal{B})x^{m-1}=0,
\end{equation}
that is,
$$
\sum_{i_2,\ldots,i_m=1}^n(a_{ii_2\cdots i_m}-\lambda b_{ii_2\cdots i_m})x_{i_2}\cdots x_{i_m}=0,~~~i=1,2,\ldots,n,$$
possesses a nonzero solution. Here, $\lambda$ is called a $\mathcal{B}$-eigenvalue of $\mathcal{A}$, and $x$ is called a $\mathcal{B}$-eigenvector of $\mathcal{A}$.
\end{definition}

With the above definition, the classical higher order {\it tensor generalized eigenvalue problem} (TGEiP) is to find a pair of $(\lambda,x)$ satisfying \eqref{ABEigen}. It is obvious that if $\mathcal{B}=\mathcal{I}$, the unit tensor $\mathcal{I}=(\delta_{i_1\cdots i_m})$, where $\delta_{i_1\cdots i_m}$ is the Kronecker symbol
$$
\delta_{i_1\cdots i_m}=\left\{
\begin{array}{ll}
1,&\;\;{\rm if~}i_1=\cdots =i_m,\\
0,&\;\;{\rm otherwise},
\end{array}
\right.
$$
then the resulting $\mathcal{B}$-eigenvalues reduce to the typical eigenvalues, and the real $\mathcal{B}$-eigenvalues with real eigenvectors are the $H$-eigenvalues, in the terminology of \cite{Q05,QSW07}.
In the literature, we have witnessed that tensors and eigenvalues/eigenvectors of tensors have fruitful applications in various fields such as magnetic resonance imaging \cite{BV08,QYW10}, higher-order Markov chains \cite{NQZ09} and
best-rank one approximation in date analysis \cite{QWW09}, whereby many nice properties such as the
Perron-Frobenius theorem for eigenvalues/eigenvectors of nonnegative square tensor have
been well established, see, e.g., \cite{CPZ08,YY10}.

In this paper, we consider the {\it tensor generalized eigenvalue complementarity problem} (TGEiCP), which can be mathematically characterized as finding a nonzero vector $\bar x\in \mathbb{R}^n$ and a scalar $\bar \lambda\in \mathbb{R}$ with property
\begin{equation}\label{GECP}
\bar x \in K,~~~\bar \lambda \mathcal{B} \bar x ^{m-1}-\mathcal{A}\bar x^{m-1}\in K^*, ~~~\langle \bar x,\bar \lambda \mathcal{B} \bar x^{m-1}-\mathcal{A}\bar x^{m-1}\rangle=0,
\end{equation}
where $\mathcal{A}$ and $\mathcal{B}$ are two given $m$-th order $n$-dimensional higher tensors, $K$ is a closed and convex cone in $\mathbb{R}^n$, and $K^*$ is the positive dual cone of $K$, i.e., $K^*:=\{w\in \mathbb{R}^n~:~\langle w,k\rangle\geq 0,\forall ~k\in K \}$. As EiCPs closely relate to differential inclusions with processes defined by linear
complementarity conditions, TGEiCPs are also closely related to a class of differential inclusions with nonconvex processes $H$ defined by
$$
{\rm Gr}( H):=\{(x,y)\in \mathbb{R}^n\times \mathbb{R}^n~|~ x \in K,  \mathcal{B}y^{m-1}-\mathcal{A}x^{m-1} \in K^*, \langle x, \mathcal{B}y^{m-1}-\mathcal{A}x^{m-1}\rangle= 0\}.
$$ The scalar $\lambda$ and the nonzero vector $x$ satisfying system (\ref{GECP}) are respectively called a $K$-eigenvalue
of $(\mathcal{A},\mathcal{B})$ and an associated $K$-eigenvector. In this situation, $(\lambda,x)$ is also called a $K$-eigenpair of $(\mathcal{A},\mathcal{B})$. The set of all eigenvalues is
called the $K$-spectrum of $(\mathcal{A},\mathcal{B})$, and it is defined by
$$\sigma_K(\mathcal{A},\mathcal{B}):=\{\lambda\in \mathbb{R}~:~\exists x\in \mathbb{R}^n\backslash\{0\},~K\ni x\perp \lambda \mathcal{B}x^{m-1}-\mathcal{A}x^{m-1}\in K^*\}.
$$
Throughout this paper one assumes that $K\cap (-K)=\{0\}$ and $\mathcal{B}x^m\neq 0$ for any $x\in K\backslash\{0\}$. If $K=\{x\in \mathbb{R}^n:x\geq 0\}$, then (\ref{GECP}) reduces to
\begin{equation}\label{SECP}
\bar x \geq 0,~~~\bar \lambda \mathcal{B} \bar x ^{m-1}-\mathcal{A}\bar x^{m-1}\geq 0, ~~~\langle \bar x,\bar \lambda \mathcal{B} \bar x^{m-1}-\mathcal{A}\bar x^{m-1}\rangle=0,
\end{equation}
which is a specialization of TGEiP. The scalar $\lambda$ and the nonzero vector $x$ satisfying system (\ref{SECP}) are called a Pareto-eigenvalue of $(\mathcal{A}, \mathcal{B})$ and an associated Pareto-eigenvector, respectively. The set of all Pareto-eigenvalues, defined by $\sigma(\mathcal{A},\mathcal{B})$, is called the Pareto-spectrum of $(\mathcal{A},\mathcal{B})$. If in addition $m=2$, the problem under consideration immediately reduces to the classical EiCP. If $\bar x\in {\rm int}(K)$ (respectively, $\bar x\in \{x\in \mathbb{R}^n:x>0\}$), then $\bar \lambda$ is called a strict $K$-eigenvalue (respectively, Pareto-eigenvalue) of $(\mathcal{A}, \mathcal{B})$. In particular, if $\mathcal{B}=\mathcal{I}$, then the $K$ (Pareto)-eigenvalue/eigenvector of $(\mathcal{A},\mathcal{B})$ is called the $K$ (Pareto)-eigenvalue/eigenvector of $\mathcal{A}$, and the $K$ (Pareto)-spectrum of $(\mathcal{A},\mathcal{B})$ is called the $K$ (Pareto)-spectrum of $\mathcal{A}$.

The main contributions of this paper are four folds. As we have mentioned in above, TGEiCP is an essential extension of EiCP. Accordingly, a natural question is that whether TGEiCP has solutions like EiCP. In this paper, we first give an affirmative answer to this question, thereby discussing the existence of the solution of TGEiCP \eqref{GECP} under some conditions. Note that TGEiCP is also a special case of complementarity problem, and it is well documented in \cite{FP03} that one of the most popular avenues to solve complementarity problems is reformulating them as optimization problems. Hence, we here also introduce two equivalent optimization reformulations of TGEiCP, which further facilitates the analysis of upper bound of cone eigenvalues of tensor. With the existence of the solution of TGEiCP, ones may be further interested in a truth that how many eigenvalues exist. Therefore, the third objective of this paper is to establish theoretical results concerning the bounds of the number of eigenvalues of TGEiCP. Finally, we develop a projection algorithm to solve TGEiCP, which is an easily implementable algorithm as long as the convex cone $K$ is simple enough in the sense that the projection onto $K$ has explicit representation. As an illustration of our theoretical results, we implement our proposed projection algorithm to solve some synthetic examples and report the corresponding computational results.

The structure of this paper is as follows. In Section \ref{SecExist}, the existence of solution for TGEiCP is discussed under some reasonable assumptions. Two optimization reformulations of TGEiCP are presented in Section \ref{Reformu}, and the relationship of TGEiCP with the optimization of the Rayleigh quotient associated to tensors is established. Moreover, based upon a reformulated optimization model, an upper bound of cone eigenvalues of tensor is also established. In Section \ref{Bound}, some theoretical results concerning the bounds of number of eigenvalues of TGEiCP are presented. To solve TGEiCP, we develop a so-called {\it scaling-and-projection algorithm} (SPA) and conduct some numerical simulations to support our results of this paper. Finally, we complete this paper with drawing some concluding remarks in Section \ref{SecCon}.

\medskip

\noindent{\bf Notation.} Let $\mathbb{R}^n$ denote the real Euclidean space of column vectors of length $n$. Denote $\mathbb{R}_+^n=\{x\in \mathbb{R}^n:x\geq 0\}$ and $\mathbb{R}_{++}^n=\{x\in \mathbb{R}^n:x> 0\}$. Let $\mathcal{A}$ be a tensor of order $m$ and dimension $n$, and $J$ be a subset of the index set $N:=\{1,2,\ldots,n\}$. We denote the principal sub-tensor of $\mathcal{A}$ by $\mathcal{A}_J$, which is obtained by homogeneous polynomial $\mathcal{A}x^m$ for all $x=(x_1,x_2,\ldots,x_n)^\top$ with $x_i=0$ for $N\backslash J$. So, $\mathcal{A}_J$ is a tensor of order $m$ and dimension $|J|$, where the symbol $|J|$ denotes the cardinality of $J$. For a vector $x\in \mathbb{R}^n$ and an integer $r\geq 0$, denote $x^{[r]}=(x_1^r,x_2^r,\ldots, x_n^r)^\top$.

\section{Existence of the solution for TGEiCP}\label{SecExist}
This section deals with the existence of the solution for TGEiCP. Let $K$ be a closed and convex pointed cone in $\mathbb{R}^n$. Recall that a nonempty set $S\subset \mathbb{R}^n$ generates a cone $K$ and write $K:={\rm cone}(S)$ if $K:=\{t s~:~s\in S,t\in \mathbb{R}_+\}$. If in addition $S$ does not contain zero and for each $k\in K\backslash \{0\}$, there exists unique $s\in S$ and $t\in \mathbb{R}_+$ such that $k=ts$, then we say that $S$ is a basis of $K$. Whenever $S$ is a finite set, ${\rm cone}({\rm conv}(S))$ is called a polyhedral cone, where ${\rm conv}(S)$ stands for the convex hull of $S$. Let $K$ be a closed convex cone equipped with a compact basis $S$. To study the existence of solution  for TGEiCP, we first make the following assumption.
\begin{assumption}\label{Assumpt1}
It holds that $\mathcal{B}x^{m}\neq 0$ for every vector $x\in S$.
\end{assumption}

\begin{remark}\label{Rem1}
It is easy to see that Assumption \ref{Assumpt1} holds if and only if one of the tensors $\mathcal{B}$ (or $-\mathcal{B}$) is strictly $K$-positive, i.e., $\mathcal{B}x^m>0$ (or $-\mathcal{B}x^m>0$) for any $x\in K\backslash\{0\}$. In particular, when $K=\mathbb{R}_+^n$, if $\mathcal{B}$ is a strictly copositive tensor (see \cite{Qi13,SQ15}), then $\mathcal{B}$ satisfies Assumption \ref{Assumpt1}. It is easy to see that if $\mathcal{B}$ is nonnegative, i.e., $\mathcal{B}\geq 0$, and there are no index subset $J$ of $N$  such that $\mathcal{B}_J$ is a zero tensor, then $\mathcal{B}x^{m}>0$ for any $x\in \mathbb{R}_+^n\backslash \{0\}$, and hence, in this case, Assumption \ref{Assumpt1} holds.
\end{remark}

From (\ref{GECP}), one knows that if $(\bar \lambda,\bar x)\in \mathbb{R}\times (\mathbb{R}^n\backslash\{0\})$ is a $K$-eigenpair of $(\mathcal{A},\mathcal{B})$, then necessarily
$$
\bar\lambda=\frac{\mathcal{A}\bar x^m}{\mathcal{B}\bar x^m},
$$
provided $\mathcal{B}\bar x^m\neq 0$. Consequently, by
the second expression of (\ref {GECP}), it holds that
$$
\frac{\mathcal{A}\bar x^{m}}{\mathcal{B}\bar x^{m}}\mathcal{B} \bar x ^{m-1}-\mathcal{A}\bar x^{m-1}\in K^*.
$$
We now present the existence theorem of TGEiCP, which is a particular instance of Theorem 3.3 in \cite{LS01}. However, for the sake of completeness, here we still present its proof.
\begin{theorem}\label{ExistTh}
Let $K$ be a cone equipped with convex compact basis $S$. If
Assumption \ref{Assumpt1} holds, then TGEiCP \eqref{GECP} has at
least one solution.
\end{theorem}

\begin{proof}
Define $F:S\times S\rightarrow \mathbb{R}$ by
\begin{equation}\label{FxyF}
F(x,y)=\langle\mathcal{A}x^{m-1},y\rangle-\frac{\mathcal{A}x^{m}}{\mathcal{B}x^{m}}\langle\mathcal{B} x ^{m-1},y\rangle.
\end{equation}
Since $\mathcal{B}x^{m}\neq 0$ for any $x\in S$, it is obvious that $F(\cdot,y)$ is lower-semicontinuous on $S$ for any fixed $y\in S$, and $F(x,\cdot)$ is concave on $S$ for any fixed $x\in S$. By the well-known Ky Fan inequality \cite{AF09}, there exists a vector $\bar x\in S\subset K\backslash\{0\}$ such that
\begin{equation}\label{eqFKF}
\sup_{y\in S}F(\bar x,y)\leq \sup_{y\in S}F(y,y).
\end{equation}
Consequently, since $F(y,y)=0$ for any $y\in S$, by (\ref{eqFKF}) it holds that $F(\bar x,y)\leq 0$ for any $y\in S$. Let $\bar \lambda=\frac{\mathcal{A}\bar x^{m}}{\mathcal{B}\bar x^{m}}$. Then, by (\ref{FxyF}), one knows that $\langle\bar \lambda\mathcal{B} \bar x ^{m-1}-\mathcal{A}\bar x^{m-1},y\rangle\geq 0
$ for any $y\in S$, which implies
\begin{equation}\label{GEPIN}
\bar \lambda\mathcal{B} \bar x ^{m-1}-\mathcal{A}\bar x^{m-1}\in K^*,
\end{equation} since for any $y\in K$ it holds that $y=t s$ for some $t\in \mathbb{R}_+$ and $s\in S$. Moreover, it is easy to know that
$$
\langle \bar x, \bar \lambda\mathcal{B} \bar x ^{m-1}-\mathcal{A}\bar x^{m-1}\rangle=0,
$$
which means, together with (\ref {GEPIN}) and the fact that $\bar x\in K\backslash\{0\}$, that $(\bar \lambda,\bar x)$ is a solution of (\ref{GECP}). We obtain the desired result and complete the proof.
\qed\end{proof}

From Theorem \ref{ExistTh}, we obtain the following corollary.
\begin{corollary}\label{Corollary1}
If $\mathcal{B}$ is strictly copositive, then \eqref{SECP} has at
least one solution.
\end{corollary}

\begin{proof}
 Take $S:=\{x\in \mathbb{R}^n_+~|~\sum_{i=1}^n x_i=1\}$. It is clear that $S$ is a convex compact base of $\mathbb{R}_+^n$. By Theorem \ref{ExistTh}, it follows that the conclusion holds. The proof is completed.
\qed\end{proof}

The following example shows that Assumption \ref{Assumpt1} is necessary to ensure the existence of the solution of TGEiCP.
\begin{example}
Let $m=2$. Consider the case where
$$\mathcal{A}=\left(
\begin{array}{ccc}
1&&3\\
4&&1
\end{array}
\right)~~~~{\rm and}~~~~\mathcal{B}=\left(
\begin{array}{ccc}
1&&0\\
0&&-1
\end{array}
\right).
$$
It is easy to see that Assumption \ref{Assumpt1} does not hold for the above two matrices. Since ${\rm det}(\lambda \mathcal{B}-\mathcal{A})=-\lambda^2-11\neq 0$ for any $\lambda\in \mathbb{R}$, we claim that the system of linear equations $(\lambda \mathcal{B}-\mathcal{A})x=0$ has only one unique solution $0$ for any $\lambda\in\mathbb{R}$, which means that $(\lambda, x)\in \mathbb{R}\times \mathbb{R}_{++}^2$ satisfying \eqref{SECP} does not exist. Moreover, we may check that $(\lambda \mathcal{B}-\mathcal{A})x\geq 0$ does not hold for any $(\lambda, x)\in \mathbb{R}\times (\mathbb{R}_+^2\backslash\{0\})$ with $x=(x_1,0)^\top$ or $x=(0,x_2)^\top$. Therefore, problem \eqref{SECP} has no solution.
\end{example}

\section{Optimization reformulations of TGEiCP}\label{Reformu}
In this section, we study two optimization reformulations of (\ref{SECP}). We begin with introducing a so-called generalized Rayleigh quotient related to tensors. For two given $m$-th order $n$ dimensional tensors $\mathcal{A}$ and $\mathcal{B}$, the related Rayleigh quotient is defined by
\begin{equation}\label{lamx}
\lambda(x)=\frac{\mathcal{A}x^m}{\mathcal{B}x^m},
\end{equation}
where $\mathcal{B}x^m\neq 0$. If $m=2$, then $\lambda(x)$ defined by (\ref{lamx}) reduces to one introduced in \cite{QJH04}. When $\mathcal{A}$ is symmetric and
$\mathcal{B}$ is symmetric and strictly copositive, it is easy to see that the gradient of
$\lambda(x)$ is
\begin{equation}\label{lambdaGrad}
\nabla \lambda(x)=\frac{m}{\mathcal{B}x^m}[\mathcal{A}x^{m-1}-\lambda(x)\mathcal{B}x^{m-1}].
\end{equation}
Notice that the expression (\ref{lambdaGrad}) of the gradient
of the Rayleigh quotient is only valid when $\mathcal{A}$ and $\mathcal{B}$ are both symmetric. Moreover, in this case, the stationary points of $\lambda(x)$ correspond to solutions of (\ref{SECP}). If either $\mathcal{A}$ or $\mathcal{B}$ is not symmetric, the above expression of $\nabla \lambda(x)$ is incorrect, and the relationship between
stationary points and solutions of the TGEiCP with $K=\mathbb{R}_+^n$ ceases to hold.

The following lemma presents two fundamental properties of the generalized
Rayleigh quotient $\lambda$ in (\ref{lamx}), whose matrix version was proposed in \cite{QJH04}. Its proof is straightforward and skipped here.
\begin{lemma}\label{ECPLemma3}
For all $x\in \mathbb{R}^n\backslash \{0\}$, the following statements hold:
\begin{itemize}
\itemindent 8pt
\item[(1).] $\lambda(\tau x)=\lambda(x)$, ~~$\forall \tau >0$;
\item[(2).] $x^\top \nabla\lambda(x)=0$.
\end{itemize}
\end{lemma}

We first consider the following optimization problem
\begin{equation}\label{OTEiCP}
\rho(\mathcal{A},\mathcal{B}):=\max_x\;\left\{\;\lambda(x)~|~x\in S\;\right\},
\end{equation}
where $\lambda(x)$ is defined in \eqref{lamx}, and the constraint set $S$ is determined by
\begin{equation}\label{Constrset}
S:=\left\{x\in \mathbb{R}_+^n~:~\sum_{i=1}^nx_i=1\right\},
\end{equation}
which is called the standard simplex in $\mathbb{R}^n$.

We generalize the result of symmetric EiCP studied in \cite{QJH04} to TGEiCP as the following proposition.
\begin{proposition}\label{proposition7} Assume that the tensors $\mathcal{A}$ and $\mathcal{B}$ are symmetric and $\mathcal{B}$ is strictly copositive. Let $\bar x$ be a stationary point of \eqref{OTEiCP}. Then $(\lambda(\bar x),\bar x)$ is a solution
of TGEiCP with $K=\mathbb{R}_+^n$.
\end{proposition}

\begin{proof}
Since $\bar x$ is a stationary solution of (\ref{OTEiCP}), from the structure of $S$, there exist $\bar \alpha\in \mathbb{R}^n$ and $\bar\beta\in \mathbb{R}$, such that
\begin{equation}\label{OptEiCP}
\left\{
\begin{array}{l}
-\nabla\lambda(\bar x)=\bar\alpha+\bar\beta e,\\
\bar\alpha\geq 0, \bar x\geq 0,\\
\bar \alpha^\top \bar x=0,\\
e^\top \bar x=1,
\end{array}
\right.
\end{equation}
where $e\in \mathbb{R}^n$ is a vector of ones. By (\ref{OptEiCP}), we know $-\bar x^\top\nabla\lambda(\bar x)=\bar\beta$, which implies, together with Lemma \ref{ECPLemma3} (2), that $\bar \beta=0$.
Consequently, from (\ref{lambdaGrad}), the first two expressions of (\ref{OptEiCP}) and the fact that $\mathcal{B}\bar x^m>0$, it holds that $\lambda(\bar x)\mathcal{B}\bar x^{m-1}-\mathcal{A}\bar x^{m-1}\geq 0$. This means, together with the fact that $\bar x\geq 0$ and $\bar x^\top(\lambda(\bar x)\mathcal{B}\bar x^{m-1}-\mathcal{A}\bar x^{m-1})=0$, that $(\lambda(\bar x),\bar x)$ is a solution of TGEiCP with $K=\mathbb{R}_+^n$. We complete the proof.
\qed\end{proof}

In what follows, we denote
$$
\lambda^{\rm max}_{\mathcal{A},\mathcal{B}}={\rm max}\left\{\lambda~:~\exists~x\in \mathbb{R}^n_+\backslash\{0\}{\rm~suct~that~}(\lambda,x)~{\rm is~a~ solution~ of ~(\ref{SECP})}\right\}
$$
for notational simplicity. Then, the following theorem characterizes the relationship between problem (\ref{OTEiCP}) and TGEiCP with $K:=\mathbb{R}_+^n$.

\begin{theorem}\label{ThforTEiCPOpt}
Let $\mathcal{A}$ and $\mathcal{B}$ be two $m$-th order $n$ dimensional symmetric tensors. If $\mathcal{B}$ is strictly copositive, then $\lambda^{\rm max}_{\mathcal{A},\mathcal{B}}=\rho(\mathcal{A},\mathcal{B})$.
\end{theorem}

\begin{proof}
It is obvious that the constrained set $\Omega$ of (\ref{OTEiCP}) is compact, and hence there exists a vector $\bar x\in \Omega$ such that $\rho(\mathcal{A},\mathcal{B})=\lambda(\bar x)$. It is clear that $\{e\}\cup\{e_i~:~i\in I(\bar x)\}$ is linearly independent since $\bar x\neq 0$, where $I(\bar x)=\{i\in N~:~\bar x_i=0\}$. Consequently, the first order optimality condition of (\ref{OTEiCP}) holds, which means that $\bar x$ is stationary point of (\ref{OTEiCP}). By Proposition \ref{proposition7}, we know that $(\lambda(\bar x), \bar x)$  is a solution of TGEiCP with $K=\mathbb{R}_+^n$. Hence, it holds that $\rho(\mathcal{A},\mathcal{B})\leq \lambda^{\rm max}_{\mathcal{A},\mathcal{B}}$.

Let $(\lambda,x)$ be a solution of TGEiCP with $K:=\mathbb{R}_+^n$, then $\lambda =\mathcal{A}x^m/\mathcal{B}x^m$. Taking $y=x/(e^\top x)$ implies that $y\in \Omega$. By Lemma \ref{ECPLemma3} (1), we know $\lambda =\mathcal{A}y^m/\mathcal{B}y^m$, which implies that $\lambda\leq \rho(\mathcal{A},\mathcal{B})$ from the definition of $\rho(\mathcal{A},\mathcal{B})$. So, we have $\lambda^{\rm max}_{\mathcal{A},\mathcal{B}}\leq\rho(\mathcal{A},\mathcal{B})$.

Therefore, we obtain the desired result and complete the proof.
\qed\end{proof}

We now study another optimization reformulation of TGEiCP with $K:=\mathbb{R}_+^n$. We consider the following optimization problem
\begin{equation}\label{OptP}
\gamma(\mathcal{A}, \mathcal{B})=\max_x\left\{\mathcal{A}x^m~\big{|}~x\in \Sigma\right\},
\end{equation}
where $\Sigma:=\{x\in \mathbb{R}^n_+~:~\mathcal{B}x^m=1\}$ is assumed to be compact.

\begin{remark}\label{Remark2}
If $\mathcal{B}$ is strictly copositive, then we claim that $\Sigma$ is compact. Indeed, if $\Sigma$ is not compact, then there exists a sequence $\{x^{(k)}\}\subset \Sigma$ such that $\|x^{(k)}\|\rightarrow \infty$ as $k\rightarrow\infty$. Taking $y^{(k)}:=x^{(k)}/\|x^{(k)}\|$ clearly shows $y^{(k)}\in \mathbb{R}_+^n$ and $\|y^{(k)}\|=1$. Without loss of generality, we may assume that there exists a vector $\bar y\in \mathbb{R}_+^n$ satisfying $\|\bar y\|=1$, such that $y^{(k)}\rightarrow \bar y$ as $k\rightarrow\infty$. On the other hand, we have $\mathcal{B}(y^{(k)})^m=1/\|x^{(k)}\|^m$, which implies $\mathcal{B}\bar y^m=0$. It contradicts to the fact that $\mathcal{B}\bar y^m>0$, since $\bar y \in \mathbb{R}_+^n\backslash\{0\}$.
\end{remark}

For TGEiCP with $K:=\mathbb{R}_+^n$ and \eqref{OptP}, we have the following theorem which can be proved by a similar way to that used in \cite{SQ13}.
\begin{theorem}\label{OpECP}
Let $\mathcal{A}$ and $\mathcal{B}$ be two $m$-th order $n$ dimensional symmetric tensors. If $\mathcal{B}x^m>0$ for any $x\in \mathbb{R}^n_+\backslash\{0\}$, then $\lambda^{\rm max}_{\mathcal{A},\mathcal{B}}=\gamma(\mathcal{A},\mathcal{B})$.
\end{theorem}

By Theorems \ref{ThforTEiCPOpt} and \ref{OpECP}, it follows that solving the largest Pareto eigenvalue of TGEiCP is an NP-hard problem in general, i.e., there are no polynomial-time algorithm for solving the largest Pareto eigenvalue of TGEiCP. In the rest of this section, based upon Theorem \ref{OpECP}, we further study the bound of Pareto eigenvalue of TGEiCP with $\mathcal{B}:=\mathcal{I}$ and $K:=\mathbb{R}_+^n$.

We denote by $\Omega^*$ the solution set of \eqref{SECP} with $\mathcal{B}:=\mathcal{I}$ and let
$$
|\lambda|^{\rm max}_{\mathcal{A}}={\rm max}\left\{\;|\lambda|\;:\;\exists~ x\in \mathbb{R}^n_+\backslash\{0\}{\rm~suct~that~}(\lambda,x)\in\Omega^*\;\right\}.
$$

\begin{theorem}\label{solutionbound}
Suppose $\mathcal{B}:=\mathcal{I}$. It holds that $$|\lambda|^{\rm max}_{\mathcal{A}}\leq {\rm min}\left\{n^{\frac{m-2}{2}}\|\mathcal{A}\|_F,\; \bar a\cdot n^{m-1}\right\},$$ where $\bar a:=\max\left\{\;|a_{i_1i_2\cdots i_m}|~:~1\leq i_1,i_2,\ldots, i_m\leq n\;\right\}$.
\end{theorem}

\begin{proof}
Let $(\lambda,x)$ be an arbitrary solution of (\ref{SECP}) with $\mathcal{B}:=\mathcal{I}$. Then it  holds that
$$
\lambda=\frac{\mathcal{A}x^m}{\sum_{i=1}^nx_i^m},
$$
which implies
$$
|\lambda|=\frac{|\mathcal{A}x^m|}{\sum_{i=1}^nx_i^m}\leq \frac{\|\mathcal{A}\|_F\|x^m\|_F}{\sum_{i=1}^nx_i^m},
$$
where $x^m:=\left(x_{i_1}x_{i_2}\cdots x_{i_m}\right)_{1\leq i_1,\ldots,i_m\leq n}$, which is an $m$-th order $n$-dimensional tensor. Since $$\|x^m\|_F^2=\sum_{i_1,i_2,\ldots,i_m=1}^n (x_{i_1}x_{i_2}\cdots x_{i_m})^2=\left(\sum_{i=1}^n x_i^2\right)^m\leq n^{m-2}\left(\sum_{i=1}^n x_i^m\right)^2,$$ we obtain
$$
|\lambda|\leq n^{\frac{m-2}{2}}\|\mathcal{A}\|_F.
$$
On the other hand, we have
$$
|\lambda|=\frac{|\mathcal{A}x^m|}{\sum_{i=1}^nx_i^m}\leq \frac{\bar a\left(\sum_{i=1}^nx_i\right)^m}{\sum_{i=1}^nx_i^m}\leq \bar a\cdot n^{m-1}.
$$
Hence we know
$$
|\lambda|\leq {\rm min}\left\{n^{\frac{m-2}{2}}\|\mathcal{A}\|_F, \bar a\cdot n^{m-1}\right\}.
$$
By the arbitrariness of $\lambda$, we obtain the desired result and complete the proof.
\qed\end{proof}

For the case where $\mathcal{B}$ is strict copositive but $\mathcal{B}\neq \mathcal{I}$, by a similar way, we may obtain
$$
|\lambda^{\max}_{\mathcal{A},\mathcal{B}}|\leq \frac{1}{N_{\min}(\mathcal{B})} \min\left\{n^{\frac{m-2}{2}}\|\mathcal{A}\|_F, \bar a\cdot n^{m-1}\right\},
$$
where $N_{\rm min}(\mathcal{B})=\min\left\{\mathcal{B}x^m~:~x\in \mathbb{R}_+^n,\;\sum_{i=1}^nx_i^m=1\right\}>0$ by Theorem 5 in \cite{Qi13}. Notice that the computation of $N_{\rm min}(\mathcal{B})$ is also NP-hard itself.
\section{Bounds for the number of Pareto eigenvalues}\label{Bound}

In this section, we study the estimation of the numbers of Pareto-eigenvalue of $(\mathcal{A}, \mathcal{B})$, where $\mathcal{A}$ and $\mathcal{B}$ are two given $m$-th order $n$-dimensional tensors.  We begin this section with some basic concepts and properties of eigenvalue/eigenvector of tensors.

It is well known that, on the left-hand side of \eqref{ABEigen}, $(\mathcal{A}-\lambda \mathcal{B})x^{m-1}$ is indeed a set of $n$ homogeneous polynomials with $n$ variables, denoted
by $\{P^{\lambda}_i(x) ~|~ 1\leq i\leq n\}$, of degree $(m-1)$. In the complex field, to study the solution set of a system of $n$ homogeneous
polynomials $(P_1,\ldots,P_n)$, in $n$ variables, the idea of the resultant ${\rm Res}(P_1,\ldots, P_n)$ is well defined and introduced in algebraic geometry literature, we refer
to the recent monograph \cite{CLO05} for more details. Applying it to our current problem, ${\rm Res}(P_1,\ldots, P_n)$ has the following properties.
\begin{proposition}\label{Prop1} We have the following results:
\begin{itemize}
\itemindent 8pt
\item[(1).] ${\rm Res}(P_1,\ldots, P_n)=0$, if and only if there exists $(\lambda,x)\in \mathbb{C}\times(\mathbb{C}^n\backslash\{0\})$ such that satisfies \eqref{ABEigen}.
\item[(2).]  The degree of $\lambda$ in ${\rm Res}(P_1,\ldots, P_n)$ is at most $n(m-1)^{n-1}$.
\end{itemize}
\end{proposition}

For the considered TGEiCP with $K=\mathbb{R}_+^n$, we present the following proposition which fully characterizes the Pareto-spectrum of TGEiCP.

\begin{proposition}\label{ECharact} Let $\mathcal{A}$ and $\mathcal{B}$  be two $m$-th order $n$-dimensional tensors. A real number $\lambda$ is Pareto-eigenvalue of $(\mathcal{A},\mathcal{B})$, if and only if there exists a nonempty subset $J\subseteq  N$ and a vector $w\in \mathbb{R}_{++}^{|J|}$ such that
\begin{equation}\label{Keigenvalue}
\mathcal{A}_Jw^{m-1}=\lambda \mathcal{B}_Jw^{m-1}
\end{equation}
and
\begin{equation}\label{KeigenIn}
\sum_{i_2,\ldots,i_m\in J}(\lambda b_{ii_2\ldots i_m}-a_{ii_2\ldots i_m})w_{i_2}\cdots w_{i_m}\geq 0, ~~~{\rm for~every~}i\in N\backslash J.
\end{equation}
In such a case, the vector $x\in \mathbb{R}^n_+$ defined by
$$
x_i=\left\{
\begin{array}{ll}
w_i,&i\in J,\\
0,&i\in N\backslash J
\end{array}
\right.
$$
is a Pareto-eigenvector of $(\mathcal{A}, \mathcal{B})$, associated to the real number $\lambda$.
\end{proposition}

\begin{proof}
It can be proved by a similar way to that used in \cite{SQ13} and we skip it here.
\qed\end{proof}

\begin{remark}\label{Remark3}
It is obvious that, in the case where $\mathcal{B}:=\mathcal{I}$, \eqref{Keigenvalue} and \eqref{KeigenIn} turn out to be
\begin{equation}\label{SKeigenvalue}
\mathcal{A}_Jw^{m-1}=\lambda w^{[m-1]}
\end{equation}
and
\begin{equation}\label{SKeigenIn}
\sum_{i_2,\ldots,i_m\in J}a_{ii_2\ldots i_m}w_{i_2}\cdots w_{i_m}\leq 0, ~~~{\rm for~every~}i\in N\backslash J,
\end{equation}
respectively. The corresponding conclusions of Pareto-eigenvalues of $\mathcal{A}$ were studied in \cite{SQ13}.
\end{remark}

By Proposition \ref{ECharact}, if $\lambda$ is Pareto-eigenvalue of $(\mathcal{A},\mathcal{B})$, then there exists a nonempty subset $J\subseteq  N$ such that $\lambda$ is a strict Pareto-eigenvalue of $(\mathcal{A}_J, \mathcal{B}_J)$. Motivated by the works on estimating the cardinality of the Pareto-spectrum of matrices \cite{See99}, we now state and prove the main results in this section.

\begin{theorem}\label{Proposition2} Let $\mathcal{A}$ and $\mathcal{B}$ be two given $m$-th order $n$-dimensional tensors. Assume that $(\mathcal{A},\mathcal{B})$ is not an identical singular pair. Then there are at most $\delta_{m,n}:=nm^{n-1}$ Pareto-eigenvalues of $(\mathcal{A}, \mathcal{B})$.
\end{theorem}

\begin{proof} It is obvious that, for every $k=0,1,\ldots,n-1$, there are $\binom{n}{n-k}$ corresponding principal
sub-tensors pair of order $m$ dimension $n-k$. Moreover, by Proposition \ref{Prop1}, we know that every principal
sub-tensors pair of order $m$ dimension $n-k$ can have at most $(n-k)(m-1)^{n-k-1}$ strict Pareto-eigenvalues. By Proposition \ref{ECharact}, in this
way one obtains the upper bound
$$
\delta_{m,n}=\sum_{k=0}^{n-1}\binom{n}{n-k}(n-k)(m-1)^{n-k-1}=nm^{n-1}.
$$
Hence proved.
\qed\end{proof}

Now we extend the above result to a more general case where $K$ is a polyhedral convex cone. A closed convex cone $K$ in $\mathbb{R}^n$ is said to be finitely
generated if there is a linear independent collection $\{c_1,c_2,\ldots,c_p\}$ of vectors in $\mathbb{R}^n$ such that
\begin{equation}\label{Kdef}
K={\rm cone}\{c_1,c_2,\ldots,c_p\}=\left\{\sum_{i=1}^p\alpha_jc_j~:~\alpha=(\alpha_1,\alpha_2,\ldots,\alpha_p)^\top\in \mathbb{R}_+^p\right\}.
\end{equation}
 It is clear that $K=\{C^\top\alpha ~|~\alpha\in \mathbb{R}_+^p\}$, where $C=[c_1,c_2,\ldots,c_p]^\top$. Moreover, it is easy to see that the dual cone of $K$, denoted by $K^*$, $K^*=\{w\in \mathbb{R}^n~|~Cw\geq 0\}$.
\begin{theorem}\label{proposition4}
Let $\mathcal{A}$ and $\mathcal{B}$ be two given $m$-th order $n$ dimensional tensors. If the closed convex cone $K$ admits representation \eqref{Kdef}, then $(\mathcal{A},\mathcal{B})$ has at most $\delta_{m,p}:=pm^{p-1}$ $K$-eigenvalues.
\end{theorem}
\begin{proof} We first prove that problem \eqref{GECP} with $K$ defined by \eqref{Kdef} is equivalent to finding a vector $\bar \alpha\in \mathbb{R}^p\backslash\{0\}$ and $\bar \lambda\in \mathbb{R}$ with property
\begin{equation}\label{ECPEq}
\bar \alpha \geq 0,~~~\bar \lambda \mathcal{D} \bar \alpha ^{m-1}-\mathcal{G}\bar \alpha^{m-1}\geq 0, ~~~\langle \bar \alpha,\bar \lambda \mathcal{D} \bar \alpha ^{m-1}-\mathcal{G}\bar \alpha^{m-1}\rangle=0,
\end{equation}
where $\mathcal{D}$ and $\mathcal{G}$ are two $m$-th order $p$-dimensional tensors, whose elements are denoted by
$$
d_{i_1i_2\ldots i_m}=\sum_{j_1,j_2,\ldots,j_m=1}^nb_{j_1j_2\ldots j_m}c_{i_1j_1}c_{i_2j_2}\cdots c_{i_mj_m}$$
and
$$g_{i_1i_2\ldots i_m}=\sum_{j_1,j_2,\ldots,j_m=1}^na_{j_1j_2\ldots j_m}c_{i_1j_1}c_{i_2j_2}\cdots c_{i_mj_m},$$
for $1\leq i_1,i_2,\ldots, i_m\leq p$, respectively.

Let $(\bar x,\bar \lambda)\in (\mathbb{R}^n\backslash\{0\})\times \mathbb{R}$ be a solution of (\ref{GECP}) with $K$ defined by (\ref{Kdef}).  Since $\bar x\in K$, by the definition of $K$, there exists a nonzero vector $\bar \alpha\in \mathbb{R}^p_+$ such that $\bar x=C^\top \bar \alpha$. Consequently, from $\bar \lambda \mathcal{B}\bar x^{m-1}-\mathcal{A}\bar x^{m-1}\in K^*$ and the expression of $K^*$, it holds that $C(\bar \lambda \mathcal{B}\bar x^{m-1}-\mathcal{A}\bar x^{m-1})\geq 0$, which implies
\begin{equation}\label{EqC}
C(\bar \lambda \mathcal{B}(C^\top \bar \alpha)^{m-1}-\mathcal{A}(C^\top \bar \alpha)^{m-1})\geq 0.
\end{equation}
By the definitions of $\mathcal{D}$ and $\mathcal{G}$, we know that (\ref {EqC}) can be equivalently written as
$$
\bar \lambda \mathcal{D} \bar \alpha ^{m-1}-\mathcal{G}\bar \alpha^{m-1}\geq 0.
$$
Moreover, it is easy to verify that $\langle \bar \alpha,\bar \lambda \mathcal{D} \bar \alpha ^{m-1}-\mathcal{G}\bar \alpha^{m-1}\rangle=0$. Conversely, if $(\bar \alpha,\bar \lambda)\in (\mathbb{R}^p\backslash\{0\})\times \mathbb{R}$ satisfies (\ref{ECPEq}), then we can prove that $(\bar x,\bar \lambda)$ with $\bar x=C^\top\bar \alpha$ satisfies (\ref{GECP}) by a similar way.

Consequently, by applying Theorem \ref{Proposition2} to the problem (\ref{ECPEq}), we know that $(\mathcal{A},\mathcal{B})$ has at most $\delta_{m,p}=pm^{p-1}$ $K$-eigenvalues. The proof is completed.
\qed\end{proof}

The above theorem shows that $\sigma_K(\mathcal{A},\mathcal{B})$ has finitely many elements in case where $K$ is a polyhedral convex cone. However, in the nonpolyhedral case the situation can be even worse. For instance, Iusem and Seeger \cite{IS07} successfully constructed a symmetric matrix $A$ (i.e., $2$-th order $n$ dimensional tensor) and a nonpolyhedral convex cone $K$ such that $\sigma_K(A, I_n)$ behaves like the Cantor ternary set, i.e., it is uncountable and totally disconnected.

In the rest of this section, we discuss the case where $\mathcal{B}:=\mathcal{I}$. We first present the following lemmas.

\begin{lemma}\label{ECPLemma1}
Let $\mathcal{A}$ be an $m$-th order $n$-dimensional nonnegative tensor, i.e., $a_{i_1\ldots i_m}\geq 0$ for $1\leq i_1,\ldots,i_m\leq n$. If $\mathcal{A}$ has two eigenvectors in $\mathbb{R}_{++}^n$, then, the corresponding eigenvalues are equal.
\end{lemma}

\begin{proof}
Let $\lambda_1$ and $\lambda_2$ be two Pareto-eigenvalues of $\mathcal{A}$, and $x, y\in \mathbb{R}_{++}^n$ the corresponding associated Pareto-eigenvectors, which means
$$
\mathcal{A}x^{m-1}=\lambda_1 x^{[m-1]}~~~{\rm and}~~~
\mathcal{A}y^{m-1}=\lambda_2 y^{[m-1]}.
$$
Since $\mathcal{A}$ is nonnegative tensor, we know that $\lambda_1,\lambda_2$ are nonnegative.  Without loss of generality, assume $\lambda_1\geq\lambda_2$. If $\lambda_1=0$, then $\lambda_2=0$. Now we assume $\lambda_1>0$. Denote \begin{equation}\label{t0def}
t_0={\rm min}\{t>0~:~ty-x\in \mathbb{R}_+^n\},
\end{equation}
which must exist since $y\in \mathbb{R}_{++}^n$. It is obvious that $t_0y-x\in \mathbb{R}_+^n$, which implies that $t_0y_i\geq x_i$ for all $i$. Consequently, since  $a_{i_1\ldots i_m}\geq 0$ for $1\leq i_1,\ldots,i_m\leq n$, by the definitions of $\mathcal{A}x^{m-1}$ and $\mathcal{A}(t_0y)^{m-1}$, one knows that
$$
t_0^{m-1}\lambda_2y^{[m-1]}-\lambda_1x^{[m-1]}=\mathcal{A}(t_0y)^{m-1}-\mathcal{A}x^{m-1}\in \mathbb{R}^n_+,
$$
which implies $$
t_0(\lambda_2/\lambda_1)^{\frac{1}{m-1}}y-x\in \mathbb{R}^n_+.
$$
By (\ref{t0def}), we know that $t_0\leq t_0(\lambda_2/\lambda_1)^{\frac{1}{m-1}}$, which implies $\lambda_1\leq \lambda_2$. Therefore, we obtain $\lambda_1=\lambda_2$ and complete the proof.
\qed\end{proof}

Let $\mathcal{A}$ be an $m$-th order $n$-dimensional tensor, we say that $\mathcal{A}$ is a $Z$-tensor, if all off-diagonal entries of $\mathcal{A}$ are nonpositive.

\begin{lemma}\label{ECPLemma2}
Let $\mathcal{A}$ be an $m$-th order $n$-dimensional tensor satisfying any of the following conditions:
(i) $-\mathcal{A}$ is a $Z$-tensor; (ii) $\mathcal{A}$ is a $Z$-tensor.
Then, $\mathcal{A}$ admits at most one strict eigenvalue.
\end{lemma}

\begin{proof}
We first consider case (i). Let $\lambda_1, \lambda_2\in \mathbb{R}$ be two strict
eigenvalues of $\mathcal{A}$, i.e., there are vectors $x,y\in \mathbb{R}_{++}^n$ such that $
\mathcal{A}x^{m-1}=\lambda_1x^{[m-1]}$ and $\mathcal{A}y^{m-1}=\lambda_2y^{[m-1]}$. Hence,
$$
(\mathcal{A}+\mu\mathcal{I})x^{m-1}=(\lambda_1+\mu)x^{[m-1]}~~~~~{\rm and}~~~~~(\mathcal{A}+\mu \mathcal{I})y^{m-1}=(\lambda_2+\mu)y^{[m-1]},
$$
where $\mu$ is any real number. Since $-\mathcal{A}$ is a $Z$-tensor, $\mathcal{A}+\mu\mathcal{I}$ is nonnegative for $\mu$ sufficiently large.  By Lemma \ref{ECPLemma1}, we obtain the equality $\lambda_1+\mu=\lambda_2+\mu$, which implies the desired conclusion.

In case (ii), the conclusion can be proved in a similar way.
\qed\end{proof}

\begin{proposition}\label{Proposition3}
Let $\mathcal{A}$ be an $m$-th order $n$-dimensional tensor satisfying any of the following conditions:
(i) $-\mathcal{A}$ is a $Z$-tensor; (ii) $\mathcal{A}$ is a $Z$-tensor.
Then, $\mathcal{A}$ can have at most $\rho_{n}:=2^{n}-1$ Pareto eigenvalues.
\end{proposition}

\begin{proof}
We only consider case (i). The conclusion for case (ii) can be proved in a similar way. For every $k=0,1,\ldots,n-1$, there are $\binom{n}{n-k}$ principal
sub-tensors of order $m$ dimension $n-k$. Since $-\mathcal{A}$ is a $Z$-tensor, it is clear that any principal sub-tensors of $-\mathcal{A}$ are also $Z$-tensor. Consequently, by Lemma \ref{ECPLemma2}, we know that, every principal sub-tensors can have at most one strict eigenvalues. Therefore, by Proposition \ref{ECharact},  one gets the upper bound
$$
\rho_{n}=\sum_{k=0}^{n-1}\binom{n}{n-k}\cdot1=2^{n}-1.
$$
We obtain the desired result and complete the proof.
\qed\end{proof}

It is easy to see that, if $\mathcal{A}$ is a nonnegative tensor, then $-\mathcal{A}$ is a $Z$-tensor. Hence, by Proposition \ref{Proposition3}, we know that any $m$-th order $n$ dimensional nonnegative tensor can have at most $(2^{n}-1)$ Pareto eigenvalues. The following example shows that the bound $\rho_n$ is sharp within the second class mentioned in Proposition \ref{Proposition3}. This is what we call the exponential growth phenomenon.

\begin{example}
Consider a $3$-th order $n$-dimensional tensor $\mathcal{A}=(a_{i_1i_2i_3})_{1\leq i_1,i_2,i_3\leq n}$ with $a_{i_1i_2i_3}=-a^{i_1+i_2+i_3}$ and $a>\sqrt[3]{4}$. Given an arbitrary index set $J=\{l_1,l_2,\ldots,l_r\}$ with $1\leq l_1< l_2< \cdots < l_r\leq n$, the principal sub-tensor
$
\mathcal{A}_J=(c_{j_1j_2j_3})_{1\leq j_1,j_2,j_3\leq r}
$
has $c_{j_1j_2j_3}=-a^{l_{j_1}+l_{j_2}+l_{j_3}}$. Take vector $\xi=(a^{\frac{l_1}{2}},a^{\frac{l_2}{2}},\ldots,a^{\frac{l_r}{2}})^\top$. It is obvious that $\xi\in \mathbb{R}^r_{++}$ and
\begin{align*}
(\mathcal{A}_J\xi^2)_j=\sum_{j_2,j_3=1}^rc_{jj_2j_3}\xi_{j_2}\xi_{j_3} &=-\sum_{j_2,j_3=1}^ra^{l_j+l_{j_2}+l_{j_3}}a^{\frac{l_{j_2}}{2}}a^{\frac{l_{j_3}}{2}} \\
&=-\left(\sum_{j\in J}a^{\frac{3}{2}j}\right)^2a^{l_j}=\lambda_J\xi^2_{j},
\end{align*}
where $\lambda_J=-\left(\sum_{j\in J}a^{\frac{3}{2}j}\right)^2$. This means that (\ref{SKeigenvalue}) holds. Since $a_{i_1i_2i_3}<0$ and $\xi>0$, one does not have to worry about the condition (\ref{SKeigenIn}). By Remark \ref{Remark3}, we know that $\lambda_J$ is a Pareto-eigenvalue of $\mathcal{A}$. Now we need to check that $\lambda_{J_1}\neq \lambda_{J_2}$ whenever $J_1\neq J_2$. Take $J_1,J_2\subseteq \{1,2,\ldots,n\}$ with $J_1\neq J_2$. Since $J_1\triangle J_2=(J_1\backslash J_2)\cup (J_2\backslash J_1)\neq\emptyset$, one can define $k={\rm max}\{k\in \{1,2,\ldots,n\},k\in J_1\triangle J_2\}$. Without loss of generality, we assume that $k\in J_2$, which implies $k\not\in J_1$. In this case, we have
$$
\sqrt{\lambda_{J_1}}-\sqrt{\lambda_{J_2}}=\sum_{j\in J_2}a^{\frac{3}{2}j}-\sum_{j\in J_1}a^{\frac{3}{2}j}=\sum_{j\in J_2,j\leq k}b^j-\sum_{j\in J_1,j\leq k-1}b^j.
$$
where $b=a^{\frac{3}{2}}$. This implies that
$$
\sqrt{\lambda_{J_1}}-\sqrt{\lambda_{J_2}}=\sum_{j\in J_2,j\leq k}b^j-\sum_{j\in J_1,j\leq k-1}b^j\geq b^k-\displaystyle\sum_{j=1}^{k-1}b^j=\frac{b^{k+1}-2b^k+b}{b-1}\geq \frac{b}{b-1}>0,
$$
where the last inequality comes the fact $b>2$ from the given condition that $a>\sqrt[3]{4}$. Therefore, we know that $\lambda_{J_1}\neq\lambda_{J_2}$. Since there are $2^n-1$ ways of choosing the index
set $J$, there are as many elements in the Pareto spectrum of this special tensor $\mathcal{A}$.
\end{example}

\begin{proposition}\label{proposition5}
Suppose that there exists an index subset $J_0\subseteq N$ with $|J_0|=l$ such that $a_{ii_2\ldots i_m}>0$ for any $i\in J_0$ and $i_2,\ldots, i_m\in N\backslash \{i\}$. Then $\mathcal{A}$ has at most $\gamma_{m,n}^l:=[n(m-1)+l](m-1)^{l-1}m^{n-l-1}$ Pareto-eigenvalues. In particular, if $J_0=N$, then $\mathcal{A}$ has at most $\mu_{m,n}:=n(m-1)^{n-1}$ Pareto-eigenvalues.
\end{proposition}

\begin{proof}
Under the given condition, we only need to consider the principal sub-tensor $\mathcal{A}_J$ with $J_0\subseteq J$, which is due to the condition (\ref{KeigenIn}). Among the principal sub-tensors of order $m$ dimension $k$, there are $\binom{n-l}{k-l}$ of them with that property. This leads to the upper bound
$$
\begin{array}{lll}
\gamma_{m,n}^l&=&\displaystyle\sum_{k=l}^n\binom{n-l}{k-l}k(m-1)^{k-1}\\
&=&\displaystyle(m-1)^l\sum_{s=0}^{n-l}\binom{n-l}{s}(s+l)(m-1)^{s-1}\\
&=&[n(m-1)+l](m-1)^{l-1}m^{n-l-1}.
\end{array}
$$
In particular, if $J_0=N$, we obtain immediately the desired result. The proof is completed.
\qed\end{proof}

A similar type of argument leads to the following result:
\begin{proposition}\label{proposition6} Suppose that there exists an index set $J_0\subseteq N$  with $|J_0|=l$ such that $a_{ii_2\ldots i_m}>0$ for any $i\in J_0$ and $i_2,\ldots, i_m\in N\backslash \{i\}$. Moreover, suppose that $-\mathcal{A}$ is a $Z$-tensor. Then,
$\mathcal{A}$ has at most $\alpha^l_{n}:=2^{n-l}$ Pareto-eigenvalues.
\end{proposition}

\begin{proof}
This time one has to compute
$$
\alpha^l_{n}=\sum_{k=l}^n\binom{n-1}{k-1}\cdot 1=\sum_{s=0}^{n-l}\binom{n-1}{s}\cdot 1=2^{n-l}.
$$
We obtain the desired result and complete the proof.
\qed\end{proof}

Theorems \ref{Proposition2}--\ref{proposition4} and Propositions \ref{Proposition3}--\ref{proposition6} extend the corresponding results for bounds of Pareto eigenvalue of square matrix, which were studied in \cite{See99}, to the case higher order tensors. In the square matrix case, i.e., $m=2$, it is clear that
$$
\alpha_n^1\leq\rho_n\leq\gamma^1_{2,n}\leq\delta_{2,n},
$$
which was presented in \cite{See99}. In the tensor case, i.e., $m\geq 3$, it is obvious that $\alpha_n^l\leq\rho_n$ and $\gamma^l_{m,n}\leq\delta_{m,n}$ for any $1\leq l\leq n$. Moreover, it is not difficult to verify that, if $l=n$ then
$
\gamma^l_{m,n}=n(m-1)^{n-1}\geq n2^{n-1}\geq \rho_n$;
if $1\leq l\leq n-1$, then $\gamma^l_{m,n}\geq [n(m-1)+1](m-1)^{n-2}\geq (2n+1)2^{n-2}\geq \rho_n$. Therefore, it always holds that
$$
\alpha_n^l\leq\rho_n\leq\gamma^l_{m,n}\leq\delta_{m,n}
$$
for any $1\leq l\leq n$.

\section{Numerical algorithm and simulations}
In this section, we first introduce an implementable algorithm for solving the TGEiCP. Then, we conduct some numerical results to verify the existence of the solution of TGEiCP and the reliability of our proposed algorithm.

\subsection{Numerical algorithm}\label{SecAlg}
It well known that the general nonlinear complementarity problem can also be transformed into a
system of equations. Therefore, it is of course possible to apply the {\it semismooth} and {\it smoothing} Newton methods to solve the problem under consideration in this paper. However, TGEiCP is more complicated than the classical EiCP due to the high-dimensional structure of tensor, thereby making such second-order algorithms difficult to be implemented. Motivated by the recent work in \cite{CS10} for solving matrix cone constrained eigenvalue problem, in this section, we extend the so-called {\it scaling-and-projection algorithm} (SPA), developed in \cite{CS10}, to solve \eqref{GECP} and follow the same name for TGEiCP. The corresponding algorithm can be described in Algorithm \ref{alg1}. Throughout this section, we assume that $\mathcal{B}$ is strictly $K$-positive, i.e., $\mathcal{B}x^m>0$ for any $x\in K\backslash\{0\}$.

\begin{algorithm}[!htb]
\caption{A Scaling-and-Projection Algorithm (SPA).}\label{alg1}
\begin{algorithmic}[1]
\STATE Take any starting point $u^{(0)}\in K\backslash\{0\}$, and define $x^{(0)}=u^{(0)}/\sqrt[m]{\B (u^{(0)})^m}$.
\FOR{$k=0,1,2,\cdots$}
\STATE One has a current point $x^{(k)}\in K\backslash\{0\}$. Compute
\[\label{step1}
\lambda_k=\frac{\A (x^{(k)})^m}{\B (x^{(k)})^m}~~~~{\rm and}~~~~y^{(k)}=\A (x^{(k)})^{m-1}-\lambda_k\B (x^{(k)})^{m-1}.\]
\STATE If $\|y^{(k)}\|=0$, then stop. Otherwise, let $s_k:=\|y^{(k)}\|$, and compute
\[\label{step2}
u^{(k)}=\Pi_K\left[x^{(k)}+s_ky^{(k)}\right]~~~~{\rm and}~~~~x^{(k+1)}=\frac{u^{(k)}}{\sqrt[m]{\B (u^{(k)})^m}}.\]
\ENDFOR
\end{algorithmic}
\end{algorithm}

It is easy to verify that iterative scheme \eqref{step1} always ensures $\langle x^{(k)},y^{(k)}\rangle=0$. As a consequence, $y^{(k)}\in K^*$ clearly means that $(x^{(k)},y^{(k)})$ is a solution of problem \eqref{GECP}. However, for the sake of convenience, we often use $\|y^{(k)}\|=0$ as the stopping condition in algorithmic framework instead of $y^{(k)}\in K^*$.


As we have mentioned, our proposed algorithm is a straightforward extension of \cite{CS10}, we can easily get the following convergence theorem. For the sake of simplicity, we skip the corresponding proof of Algorithm \ref{alg1}, those who are interested in are referred to \cite{CS10} for a similar proof.

\begin{theorem}\label{conveg_th}
Let the sequence $\{x^{(k)}\}$ be generated by Algorithm \ref{alg1} and further satisfy $\mathcal{B}(x^{(k)})^m=1$. Assume convergence of $\{x^{(k)}\}$ toward some limit that one denotes by $\bar x$. Then,
\begin{equation}\label{Barxy}
\lim_{k\rightarrow \infty}\lambda_k=\bar \lambda:=\frac{\mathcal{A}\bar x^m}{\mathcal{B}\bar x^m},~~~~\lim_{k\rightarrow \infty}y^{(k)}=\bar y:=\mathcal{A}\bar x^{m-1}-\bar \lambda\mathcal{B}\bar x^{m-1},
\end{equation}
and $(\bar \lambda, \bar x)$ is a solution of \eqref{GECP}.
\end{theorem}

\begin{remark}\label{Remark4}
As mentioned in \cite{CS10}, if $K$ has a complicated structure, then computing $u^{(k)}$ in Algorithm \ref{alg1} is not an easy task. 
However, there are many interesting cones for
which the projection map admits an explicit and easily computable formula. This
is true, for instance, for the Pareto cone, for the Loewner cone
of positive semidefinite symmetric matrices, for the Lorentz cone and,
more generally, for any revolution cone. Therefore Algorithm \ref{alg1} is easily implemented as long as the projection onto $K$ is easy enough to computed explicitly.
\end{remark}

\begin{remark}\label{Remark5}
The tensors $\mathcal{A}$ and $\mathcal{B}$ considered above are not necessarily symmetric. If $K=\mathbb{R}_+^n$ and the tensors $\mathcal{A}$ and $\mathcal{B}$ are both symmetric, then the symmetric TGEiCP can be solved by computing a stationary point of the nonlinear
program \eqref{OTEiCP}. The constraint set of this program is the simplex $S$ defined
by \eqref{Constrset}. The special structure of this set $S$ makes the computation of projections of vectors over
$S$ very easy. On the other hand, the objective function of the required nonlinear program has
Hessian whose computation is quite involved. These features lead to our decision of investigating
first order algorithms that are based on gradients and projections. This will be our investigation task in future.
\end{remark}

\subsection{Numerical simulations}
We have theoretically discussed the existence of the solution of TGEiCP in Section \ref{SecExist} and introduced an implementable projection method to solve the problem under consideration in Section \ref{SecAlg}. Thus, in this section, we aim at verifying that our theoretical results are true, in addition to demonstrating the reliability of the proposed algorithm. We implement Algorithm \ref{alg1} by {\sc Matlab} R2012b and conduct the numerical simulations on a Lenovo notebook with Inter(R) Core(TM) i5-2410M CPU
2.30GHz and 4GB RAM running on Windows 7 Home Premium operating system.

In our experiments, we concentrate on three concrete TGEiCPs with symmetric structure and only list the details of tensors $\A$ and $\B$ in the ensuing examples.
\begin{example}\label{exam1}
We consider two $4$-th order $2$-dimensional symmetric tensors $\A$ and $\B$, where the tensor $\A$ is specified as
$$\A(:,:,1,1)=\lb\begin{array}{ccc}
0.8147 && 0.5164 \\ 0.5164 && 0.9134
\end{array}\rb,\quad \A(:,:,1,2)=\lb\begin{array}{ccc}
0.4218 &&  0.8540 \\    0.8540  && 0.9595
\end{array}\rb, $$
$$\A(:,:,2,1)=\lb\begin{array}{ccc}
 0.4218 && 0.8540 \\    0.8540  && 0.9595
\end{array}\rb,\quad \A(:,:,1,2)=\lb\begin{array}{ccc}
  0.6787 && 0.7504 \\    0.7504 && 0.3922
\end{array}\rb, $$
and the tensor $\B$ is specified as
$$\B(:,:,1,1)=\lb\begin{array}{ccc}
  1.6324 &&  1.1880 \\    1.1880  && 1.5469
\end{array}\rb,\quad \B(:,:,1,2)=\lb\begin{array}{ccc}
    1.6557 && 1.4424 \\    1.4424  && 1.9340
\end{array}\rb, $$
$$\B(:,:,2,1)=\lb\begin{array}{ccc}
1.6557 && 1.4424 \\   1.4424  && 1.9340
\end{array}\rb,\quad \B(:,:,1,2)=\lb\begin{array}{ccc}
    1.6555 && 1.4386\\    1.4386 && 1.0318
\end{array}\rb. $$
\end{example}

\begin{example}\label{exam2}
This example considers two $4$-th order $3$-dimensional symmetric tensors $\A$ and $\B$, where $\A$ and $\B$ are specified as follows:
$$\A(:,:,1,1)=\lb\begin{array}{ccccc}
  0.6229  &&  0.2644  &&  0.3567 \\
    0.2644  &&  0.0475  && 0.7367\\
    0.3567  &&  0.7367 &&   0.1259
\end{array}\rb,\; \A(:,:,1,2)=\lb\begin{array}{ccccc}
  0.7563  &&  0.5878  &&  0.5406\\
    0.5878  &&  0.1379  &&  0.0715\\
    0.5406  &&  0.0715  &&  0.3725
\end{array}\rb, $$
$$\A(:,:,1,3)=\lb\begin{array}{ccccc}
    0.0657  &&  0.4918  &&  0.9312\\
    0.4918  &&  0.7788  &&  0.9045\\
    0.9312  &&  0.9045  &&  0.8711
\end{array}\rb,\; \A(:,:,2,1)=\lb\begin{array}{ccccc}
    0.7563  && 0.5878  &&  0.5406\\
    0.5878  &&  0.1379 &&  0.0715\\
    0.5406  &&  0.0715 &&  0.3725
\end{array}\rb, $$
$$\A(:,:,2,2)=\lb\begin{array}{ccccc}
    0.7689  &&  0.3941  &&  0.6034\\
    0.3941  &&  0.3577  &&  0.3465\\
    0.6034  &&  0.3465  &&  0.4516
\end{array}\rb,\; \A(:,:,2,3)=\lb\begin{array}{ccccc}
    0.8077 &&   0.4910  &&  0.2953\\
    0.4910 &&   0.5054  &&  0.5556\\
    0.2953 &&   0.5556  &&  0.9608
\end{array}\rb, $$
$$\A(:,:,3,1)=\lb\begin{array}{ccccc}
     0.0657  &&  0.4918  &&  0.9312\\
    0.4918  &&  0.7788 &&   0.9045\\
    0.9312  &&  0.9045 &&   0.8711
\end{array}\rb,\; \A(:,:,3,2)=\lb\begin{array}{ccccc}
 0.8077  &&  0.4910  &&  0.2953\\
    0.4910  &&  0.5054  &&  0.5556\\
    0.2953  &&  0.5556  &&  0.9608
\end{array}\rb, $$
$$\A(:,:,3,3)=\lb\begin{array}{ccccc}
    0.7581 &&   0.7205 &&   0.9044\\
    0.7205 &&   0.0782 &&   0.7240\\
    0.9044 &&   0.7240 &&   0.3492
\end{array}\rb,\; \B(:,:,1,1)=\lb\begin{array}{ccccc}
    0.6954  &&  0.4018  &&  0.1406\\
    0.4018  &&  0.9957  &&  0.0483\\
    0.1406  &&  0.0483  &&  0.0988
\end{array}\rb, $$
$$\B(:,:,1,2)=\lb\begin{array}{ccccc}
    0.6730 &&   0.5351 &&   0.4473\\
    0.5351 &&   0.2853 &&   0.3071\\
    0.4473 &&   0.3071 &&   0.9665
\end{array}\rb,\; \B(:,:,1,3)=\lb\begin{array}{ccccc}
    0.7585 &&   0.6433 &&  0.2306\\
    0.6433 &&   0.8986 &&   0.3427\\
    0.2306 &&   0.3427 &&   0.5390
\end{array}\rb, $$
$$\B(:,:,2,1)=\lb\begin{array}{ccccc}
    0.6730  &&  0.5351 &&   0.4473\\
    0.5351  &&  0.2853 &&   0.3071\\
    0.4473  &&  0.3071 &&   0.9665
\end{array}\rb,\; \B(:,:,2,2)=\lb\begin{array}{ccccc}
    0.3608  &&   0.3914 &&   0.5230\\
    0.3914  &&  0.6822  &&  0.5516\\
    0.5230  &&  0.5516  &&  0.7091
\end{array}\rb, $$
$$\B(:,:,2,3)=\lb\begin{array}{ccccc}
    0.4632 &&   0.2043 &&   0.2823\\
    0.2043 &&   0.7282 &&   0.7400\\
    0.2823 &&   0.7400 &&   0.9369
\end{array}\rb,\; \B(:,:,3,1)=\lb\begin{array}{ccccc}
    0.7585 &&   0.6433 &&   0.2306\\
    0.6433 &&   0.8986 &&   0.3427\\
    0.2306 &&   0.3427 &&   0.5390
\end{array}\rb, $$
$$\B(:,:,3,2)=\lb\begin{array}{ccccc}
    0.4632  &&  0.2043 &&   0.2823\\
    0.2043  &&  0.7282 &&   0.7400\\
    0.2823  &&  0.7400 &&   0.9369
\end{array}\rb,\; \B(:,:,3,3)=\lb\begin{array}{ccccc}
    0.8200  &&  0.5914  &&  0.4983\\
    0.5914  &&  0.0762  &&  0.2854\\
    0.4983  &&  0.2854  &&  0.1266
\end{array}\rb.$$
\end{example}

\begin{example}\label{exam3}
This example also considers two $4$-th order $3$-dimensional symmetric tensors $\A$ and $\B$, where $\A$ and $\B$ take their components as follows:
$$\A(:,:,1,1)=\lb\begin{array}{ccccc}
    0.4468  &&  0.4086  &&  0.5764\\
    0.4086  &&  0.8176  &&  0.5867\\
    0.5764  &&  0.5867  &&  0.8116
\end{array}\rb,\; \A(:,:,1,2)=\lb\begin{array}{ccccc}
    0.2373  &&  0.5028  &&  0.7260\\
    0.5028  &&  0.5211  &&  0.4278\\
    0.7260  &&  0.4278  &&  0.6791
\end{array}\rb, $$
$$\A(:,:,1,3)=\lb\begin{array}{ccccc}
    0.0424 &&   0.0841  &&  0.6220\\
    0.0841 &&   0.8181  &&  0.4837\\
    0.6220 &&   0.4837  &&  0.6596
\end{array}\rb,\; \A(:,:,2,1)=\lb\begin{array}{ccccc}
    0.2373 &&   0.5028 &&   0.7260\\
    0.5028 &&   0.5211 &&   0.4278\\
    0.7260 &&   0.4278 &&   0.6791
\end{array}\rb, $$
$$\A(:,:,2,2)=\lb\begin{array}{ccccc}
    0.3354 &&   0.7005 &&   0.3154\\
    0.7005 &&   0.1068 &&   0.7164\\
    0.3154 &&   0.7164 &&   0.7150
\end{array}\rb,\; \A(:,:,2,3)=\lb\begin{array}{ccccc}
    0.1734 &&   0.5972 &&   0.6791\\
    0.5972 &&   0.0605 &&   0.4080\\
    0.6791 &&   0.4080 &&   0.6569
\end{array}\rb, $$
$$\A(:,:,3,1)=\lb\begin{array}{ccccc}
    0.0424  &&  0.0841  &&  0.6220\\
    0.0841  &&  0.8181  &&  0.4837\\
    0.6220  &&  0.4837  &&  0.6596
\end{array}\rb,\; \A(:,:,3,2)=\lb\begin{array}{ccccc}
    0.1734 &&   0.5972 &&   0.6791\\
    0.5972 &&   0.0605 &&   0.4080\\
    0.6791 &&   0.4080 &&   0.6569
\end{array}\rb, $$
$$\A(:,:,3,3)=\lb\begin{array}{ccccc}
    0.4897 &&   0.6299 &&   0.6104\\
    0.6299 &&   0.0527 &&   0.5803\\
    0.6104 &&   0.5803 &&   0.5479
\end{array}\rb,\; \B(:,:,1,1)=\lb\begin{array}{ccccc}
    2.5328 &&   2.6133 &&   2.7630\\
    2.6133 &&   2.5502 &&   2.4151\\
    2.7630 &&   2.4151 &&   2.3012
\end{array}\rb, $$
$$\B(:,:,1,2)=\lb\begin{array}{ccccc}
    2.3955 &&   2.2026 &&   2.8921\\
    2.2026 &&   2.8852 &&   2.5060\\
    2.8921 &&   2.5060 &&   2.2619
\end{array}\rb,\; \B(:,:,1,3)=\lb\begin{array}{ccccc}
    2.5186 &&   2.8867 &&   2.7372\\
    2.8867 &&   2.4538 &&   2.2579\\
    2.7372 &&   2.2579 &&   2.1332
\end{array}\rb, $$
$$\B(:,:,2,1)=\lb\begin{array}{ccccc}
    2.3955 &&   2.2026  &&  2.8921\\
    2.2026 &&   2.8852  &&  2.5060\\
    2.8921 &&   2.5060  &&  2.2619
\end{array}\rb,\; \B(:,:,2,2)=\lb\begin{array}{ccccc}
    2.9037 &&   2.7948 &&   2.5391\\
    2.7948 &&   2.1978 &&   2.2653\\
    2.5391 &&   2.2653 &&   2.4799
\end{array}\rb, $$
$$\B(:,:,2,3)=\lb\begin{array}{ccccc}
    2.6280  &&  2.1537 &&   2.2689\\
    2.1537  &&  2.9841 &&   2.2698\\
    2.2689  &&  2.2698 &&   2.1981
\end{array}\rb,\; \B(:,:,3,1)=\lb\begin{array}{ccccc}
    2.5186  &&  2.8867 &&   2.7372\\
    2.8867  &&  2.4538 &&   2.2579\\
    2.7372  &&  2.2579 &&   2.1332
\end{array}\rb, $$
$$\B(:,:,3,2)=\lb\begin{array}{ccccc}
    2.6280  &&  2.1537 &&   2.2689\\
    2.1537  &&  2.9841 &&   2.2698\\
    2.2689  &&  2.2698 &&   2.1981
\end{array}\rb,\; \B(:,:,3,3)=\lb\begin{array}{ccccc}
    2.9427  &&  2.3596  &&  2.7611\\
    2.3596  &&  2.7011  &&  2.6822\\
    2.7611  &&  2.6822  &&  2.6665
\end{array}\rb.$$
\end{example}

Note that the stopping criterion in Algorithm \ref{alg1} is $\|y^{(k)}\|=0$ for exactly solving TGEiCP. In practical implementation, we usually use
\[\label{error}{\rm RelErr:}=\|y^{(k)}\|:=\|\A (x^{(k)})^{m-1}-\lambda_k\B (x^{(k)})^{m-1}\|\leq {\rm Tol}\]
as the termination criterion to pursue an approximate solution with a preset tolerance `${\rm Tol}$'. Now, we test three scenarios of `${\rm Tol}$' by setting ${\rm Tol}:=\left\{5\cdot 10^{-3}\right.$, $10^{-3}$, $\left.5\cdot 10^{-4}\right\}$. We consider two cases of the starting point $u^{(0)}$, where the first case is a vector of ones, i.e., $u^{(0)}=(1,\cdots,1)^\t$, and the second one is a random vector uniformly distributed in $(0,1)$, (the corresponding {\sc Matlab} script is \verb"rand(n,1)"). To demonstrate the reliability of Algorithm \ref{alg1}, we report the number of iterations (`Iter.'), computing time in seconds (`Time'), the relative error (`RelErr') defined by \eqref{error}, eigenvalue (`EigValue') and the corresponding eigenvector (`EigVector'). The computational results with respect to different initial points are summarized in Tables \ref{table1} and \ref{table2}, respectively.

\begin{table}[!htb]
\begin{center}
\caption{Computational results with starting point $(1,\cdots,1)^\t$.}\vskip 0.2mm
\label{table1}
\def\temptablewidth{1\textwidth}
\begin{tabular*}{\temptablewidth}{@{\extracolsep{\fill}}lcccccc}\toprule
Example & Tol & Iter. & Time & RelErr & EigValue & EigVector\\\midrule
Example \ref{exam1}& 5.0e-03 & 657 & 0.17 & 5.007e-03 & 0.4859 & $(0.2697,0.6407)^\t$ \\
Example \ref{exam2}& 5.0e-03 & 231 & 0.06 & 5.006e-03 & 1.5609 & $(0.2168,0.1532,0.8774)^\t$ \\
Example \ref{exam3}& 5.0e-03 & 536 & 0.14 & 5.005e-03 & 0.2189 & $(0.0630,0.0000,0.7236)^\t$ \\   \midrule
Example \ref{exam1}& 1.0e-03 & 3211 & 0.78 & 1.000e-03 & 0.4850 & $(0.2601,0.6512)^\t$ \\
Example \ref{exam2}& 1.0e-03 & 1703 & 0.44 & 1.001e-03 & 1.5512 & $(0.2194,0.1576,0.8683)^\t$ \\
Example \ref{exam3}& 1.0e-03 & 2584 & 0.65 & 1.000e-03 & 0.2173 & $(0.0542,0.0000,0.7319)^\t$ \\ \midrule
Example \ref{exam1}& 5.0e-04 & 6367 & 1.61 & 5.001e-04 & 0.4849 & $(0.2589,0.6525)^\t$ \\
Example \ref{exam2}& 5.0e-04 & 2929 & 0.77 & 5.002e-04 & 1.5514 & $(0.2199,0.1575,0.8678)^\t$ \\
Example \ref{exam3}& 5.0e-04 & 5293 & 1.52 & 5.000e-04 & 0.2171 & $(0.0530,0.0000,0.7330)^\t$ \\ \bottomrule
\end{tabular*}
\end{center}
\end{table}

\begin{table}[!htb]
\begin{center}
\caption{Computational results with a random starting point.}\vskip 0.2mm
\label{table2}
\def\temptablewidth{1\textwidth}
\begin{tabular*}{\temptablewidth}{@{\extracolsep{\fill}}lcccccc}\toprule
Example & Tol & Iter. & Time & RelErr & EigValue & EigVector\\\midrule
Example \ref{exam1}& 5.0e-03 & 277 & 0.09 & 5.005e-03 & 0.4859 & $(0.2697,0.6407)^\t$ \\
Example \ref{exam2}& 5.0e-03 & 291 & 0.09 & 5.003e-03 & 1.5464 & $(0.2172,0.1600,0.8673)^\t$ \\
Example \ref{exam3}& 5.0e-03 & 519 & 0.14 & 5.003e-03 & 0.2189 & $(0.0623,0.0008,0.7234)^\t$ \\  \midrule
Example \ref{exam1}& 1.0e-03 & 3218 & 0.80 & 1.000e-03 & 0.4850 & $(0.2601,0.6512)^\t$ \\
Example \ref{exam2}& 1.0e-03 & 1613 & 0.43 & 1.001e-03 & 1.5511 & $(0.2195,0.1577,0.8680)^\t$ \\
Example \ref{exam3}& 1.0e-03 & 2636 & 0.70 & 1.000e-03 & 0.2173 & $(0.0542,0.0000,0.7319)^\t$ \\  \midrule
Example \ref{exam1}& 5.0e-04 & 6071 & 1.54 & 5.000e-04 & 0.4847 & $(0.2565,0.6551)^\t$ \\
Example \ref{exam2}& 5.0e-04 & 2341 & 0.64 & 5.002e-04 & 1.5510 & $(0.2203,0.1576,0.8672)^\t$ \\
Example \ref{exam3}& 5.0e-04 & 5341 & 1.41 & 5.001e-04 & 0.2171 & $(0.0530,0.0000,0.7330)^\t$ \\  \bottomrule
\end{tabular*}
\end{center}
\end{table}

From the data reported in Tables \ref{table1} and \ref{table2}, it is clear that our Algorithm \ref{alg1} can successfully solve the TGEiCP, even though it seems that the number of iterations increases significantly as the decrease of tolerance `${\rm Tol}$'. Actually, we tested a series of random starting points, and observed that random starting points often perform better than the deterministic vector of ones in terms of taking less iterations as reported in Table \ref{table2}. However, all experiments show that Algorithm \ref{alg1} is reliable for solving TGEiCP.

Taking a revisit on Algorithm \ref{alg1}, the iterative scheme \eqref{step2} plays an significant role in the whole algorithm. In other words, the projection step given in \eqref{step2} dominates the main task of Algorithm \ref{alg1}. As we know, the typical projection methods consist of two important components, i.e., step size and search direction. In Algorithm \ref{alg1}, $s_k$ and $y^{(k)}$ serve as the step size and search direction, respectively. It is well known that good choices of step size and search direction may lead to promising numerical performance. Turn our attention to \eqref{step2}, it can be easily seen that step size $s_k$ approaches to zero as the sequence $\{x^{(k)}\}$ gets close to a solution of TGEiCP, thereby reducing the speed of convergence of Algorithm \ref{alg1}. A naturally simple idea is to increase $s_k$ by attaching a larger constant $\alpha$ to it, that is, the projection step in \eqref{step2} turns out to be
\[\label{projnew}u^{(k)}=\Pi_K\left[x^{(k)}+\alpha s_ky^{(k)}\right].\]
In our experiments, we observe that Algorithm \ref{alg1} could be accelerated greatly when we set $\alpha\in(1,8)$.
We also report some computational results in Table \ref{table3}.

\begin{table}[!htb]
\begin{center}
\caption{Computational results with starting point $(1,\cdots,1)^\t$ and $\alpha=5$ in \eqref{projnew}.}\vskip 0.2mm
\label{table3}
\def\temptablewidth{1\textwidth}
\begin{tabular*}{\temptablewidth}{@{\extracolsep{\fill}}lcccccc}\toprule
Example & Tol & Iter. & Time & RelErr & EigValue & EigVector\\\midrule
Example \ref{exam1}& 5.0e-03 & 130 & 0.04 & 5.012e-03 & 0.4859 & $(0.2696,0.6407)^\t$ \\
Example \ref{exam2}& 5.0e-03 & 62 & 0.02 & 5.006e-03 & 1.5472 & $(0.2168,0.1597,0.8682)^\t$ \\
Example \ref{exam3}& 5.0e-03 & 105 & 0.03 & 5.010e-03 & 0.2189 & $(0.0629,0.0000,0.7236)^\t$ \\ \midrule
Example \ref{exam1}& 1.0e-03 & 639 & 0.17 & 1.001e-03 & 0.4850 & $(0.2601,0.6512)^\t$ \\
Example \ref{exam2}& 1.0e-03 & 230 & 0.07 & 1.001e-03 & 1.5501 & $(0.2204,0.1580,0.8664)^\t$ \\
Example \ref{exam3}& 1.0e-03 & 513 & 0.13 & 1.001e-03 & 0.2173 & $(0.0542,0.0000,0.7319)^\t$ \\
\midrule
Example \ref{exam1}& 5.0e-04 & 1270 & 0.32 & 5.002e-04 & 0.4849 & $(0.2589,0.6525)^\t$ \\
Example \ref{exam2}& 5.0e-04 & 549 & 0.14 & 5.003e-04 & 1.5513 & $(0.2207,0.1574,0.8669)^\t$ \\
Example \ref{exam3}& 5.0e-04 & 1054 & 0.28 & 5.003e-04 & 0.2171 & $(0.0530,0.0000,0.7330)^\t$ \\  \midrule
Example \ref{exam1}& 1.0e-04 & 6297 & 1.55 & 1.000e-04 & 0.4848 & $(0.2579,0.6536)^\t$ \\
Example \ref{exam2}& 1.0e-04 & 3227 & 0.84 & 1.000e-04 & 1.5520 & $(0.2203,0.1571,0.8679)^\t$ \\
Example \ref{exam3}& 1.0e-04 & 6332 & 1.65 & 1.000e-04 & 0.2170 & $(0.0518,0.0005,0.7337)^\t$ \\\bottomrule
\end{tabular*}
\end{center}
\end{table}

By comparing the results in Tables \ref{table1} and \ref{table3}, it is apparent that the refined projection step \eqref{projnew} outperforms the original one in \eqref{step2} in terms of taking much less iterations. In Fig. \ref{fig1}, we further consider two different projection steps, and graphically plot the evolutions of the relative error defined by \eqref{error} in the logarithmic sense, i.e., $\log(\|y^{(k)}\|)$, with respect to iterations, where the stopping tolerance `Tol' is set to be ${\rm Tol:}=10^{-4}$.

\begin{figure}[!htbp]
\includegraphics[width=0.49\textwidth]{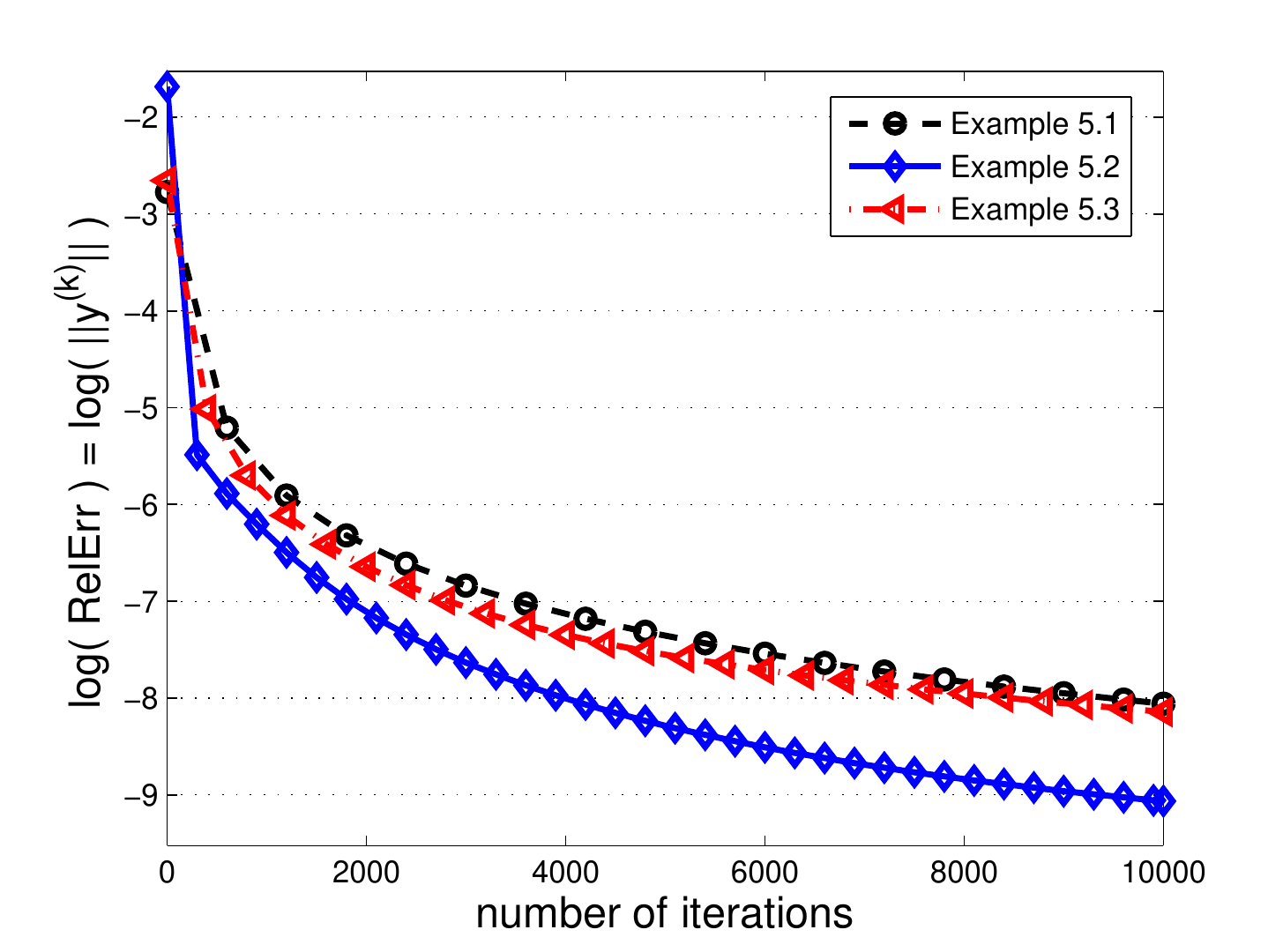}
\includegraphics[width=0.49\textwidth]{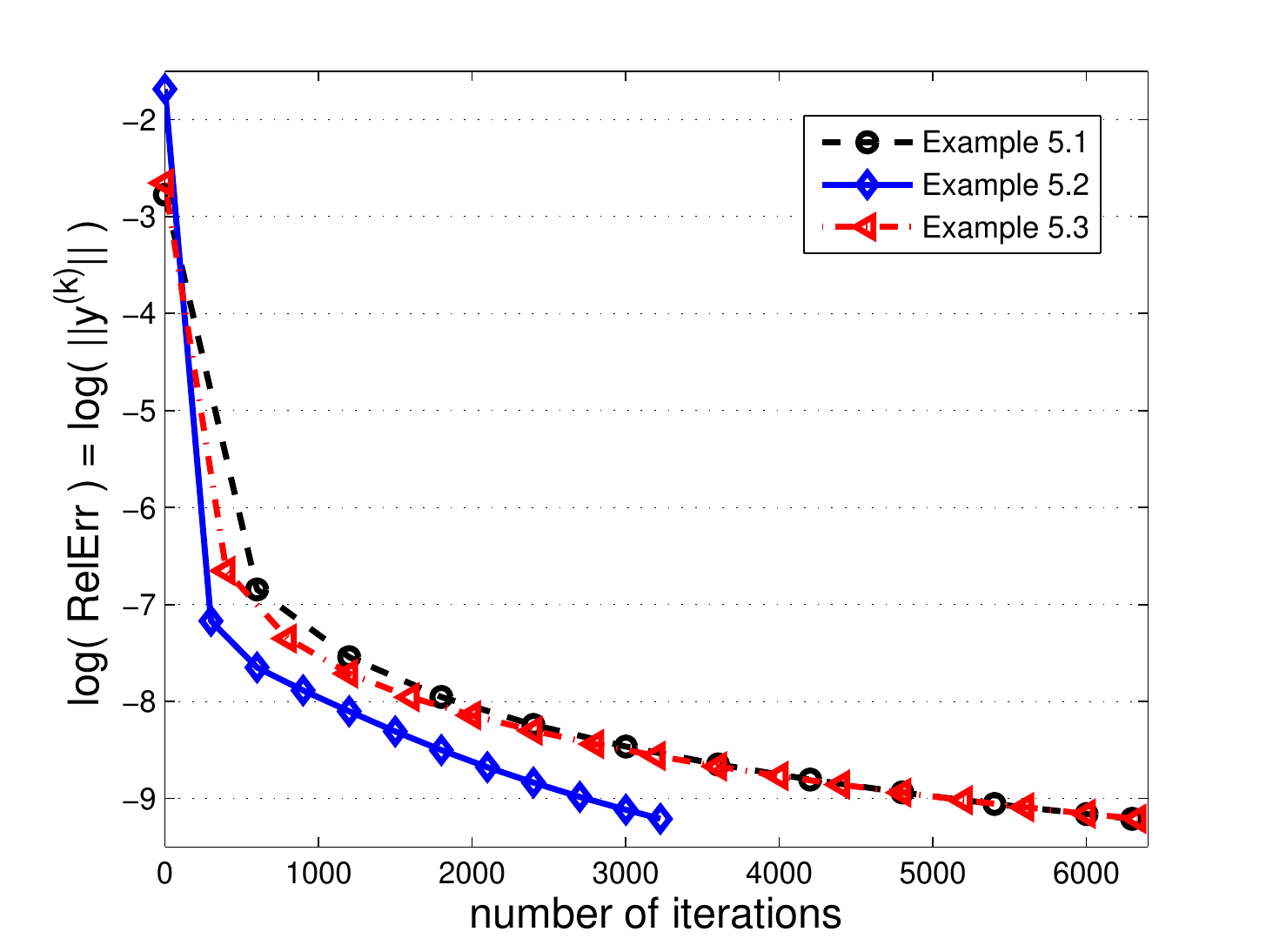}
\caption{Evolutions of `RelErr' defined by \eqref{error} with respect to iterations. The left plot corresponds to the original projection scheme, i.e., $\alpha=1$. The right one is corresponding to \eqref{projnew} with $\alpha=5$.}\label{fig1}
\end{figure}

It is clear from the above results that attaching a relaxation factor $\alpha$ in \eqref{projnew} is necessary to improve the numerical performance of our algorithm. In future work, we will introduce a self-adaptive strategy to adjust the relaxation factor $\alpha$ for an acceleration of the proposed method.

\section{Conclusions}\label{SecCon}
This paper considers the TGEiCP with symmetric structure, which is an interesting generalization of matrix eigenvalue complementarity problem. To the best of our knowledge, the development of TGEiCP is in its infancy and such a problem has been received much less attention. In this paper, we discuss the existence of the solution of TGEiCP under some conditions, in addition to presenting two equivalent optimization reformulations for the purpose of analyzing the upper bound of cone eigenvalues of tensors. The bounds of the number of eigenvalues of TGEiCP are also presented. Finally, we develop a first-order projection method which might be a better candidate for TGEiCP than second-order solvers. Note that we only consider the optimization reformulations of symmetric tensors, and many problems are lack of such a symmetric structure. Hence, our future work will further study TGEiCPs in absence of symmetric property. On the other hand, our numerical simulations show us that the attached $\alpha$ in \eqref{projnew} is important for algorithmic acceleration. Then, how to improve the numerical performance of Algorithm \ref{alg1} is also one of our future concerns.

\medskip
\begin{acknowledgements}
The first two authors were supported by National Natural Science Foundation of China (NSFC) at Grant Nos. (11171083, 11301123) and the Zhejiang Provincial NSFC at Grant No. LZ14A010003. The third author was supported by the Hong Kong Research Grant Council (Grant Nos. PolyU 502510, 502111, 501212 and 501913).
\end{acknowledgements}


\end{document}